\documentclass[11pt,a4paper,reqno]{amsart}
\usepackage{amssymb,amsmath}
\usepackage{graphics}
\usepackage{epsfig}
\usepackage{yfonts}
\usepackage{graphicx}
\usepackage[matrix,arrow,tips,curve,ps]{xy}
\usepackage{enumerate, amssymb}
\usepackage{hyperref}
\usepackage{comment}
\usepackage{paralist}
\usepackage{color}

\def\bdi{\begin{diagram}}
\def\edi{\end{diagram}}

\newtheorem{thm}{Theorem}[section]
\newtheorem{cor}[thm]{Corollary}
\newtheorem{lem}[thm]{Lemma}
\newtheorem{prop}[thm]{Proposition}

\theoremstyle{definition}
\newtheorem{defi}[thm]{Definition}
\newtheorem{defis}[thm]{Definitions}
\newtheorem{conj}[thm]{Conjecture}

\newtheorem{conv}[thm]{Convention}
\newtheorem{nota}[thm]{Notation}
\newtheorem{rem}[thm]{Remark}
\newtheorem{rems}[thm]{Remarks}
\newtheorem{exa}[thm]{Example}
\newtheorem{exas}[thm]{Examples}

\newcommand{\rien}[1]{}

\newcommand{\N}{\ensuremath{\mathbb{N}}}

\newcommand{\proj}{\ensuremath{\mathbb{P}}}

\newcommand{\cO}{{\ensuremath{\mathcal{O}}}}

\def\PP{{\mathbb P}}

\def\codim{\mathop{\rm codim}}

\renewcommand{\epsilon}{\varepsilon}
\renewcommand{\phi}{\varphi}

\newcommand{\bnum}{\begin{enumerate}}
\newcommand{\enum}{\end{enumerate}}
\renewcommand{\emptyset}{\varnothing}
\newcommand{\brem}{\begin{rem}}
\newcommand{\brems}{\begin{rems}}
\newcommand{\erem}{\end{rem}}
\newcommand{\erems}{\end{rems}}
\newcommand{\bexa}{\begin{exa}}
\newcommand{\bexas}{\begin{exas}}
\newcommand{\eexa}{\end{exa}}
\newcommand{\eexas}{\end{exas}}
\newcommand{\bdefi}{\begin{defi}}
\newcommand{\edefi}{\end{defi}}
\newcommand{\bdefis}{\begin{defis}}
\newcommand{\edefis}{\end{defis}}
\newcommand{\bcor}{\begin{cor}}
\newcommand{\ecor}{\end{cor}}
\newcommand{\blem}{\begin{lem}}
\newcommand{\elem}{\end{lem}}
\newcommand{\bconv}{\begin{conv}}
\newcommand{\econv}{\end{conv}}
\newcommand{\bconj}{\begin{conj}}
\newcommand{\econj}{\end{conj}}
\newcommand{\bprop}{\begin{prop}}
\newcommand{\eprop}{\end{prop}}
\newcommand{\bthm}{\begin{thm}}
\newcommand{\ethm}{\end{thm}}
\newcommand{\bnota}{\begin{nota}}
\newcommand{\enota}{\end{nota}}
\newcommand{\bsit}{\begin{sit}}
\newcommand{\esit}{\end{sit}}
\newcommand{\be}{\begin{eqnarray}}
\newcommand{\ee}{\end{eqnarray}}
\newcommand{\bproof}{\begin{proof}}
\newcommand{\eproof}{\end{proof}}
\def\ba{\begin{array}}
\def\ea{\end{array}}

\title[Gaps for geometric genera]{
Gaps for geometric genera}
\author{C.~Ciliberto, F.~Flamini, M.~Zaidenberg}
\address{Dipartimento di Matematica, Universit\`a degli
Studi di Roma ``Tor Vergata'', Via della Ricerca Scientifica,
00133 Roma, Italy} \email{cilibert@mat.uniroma2.it,
flamini@mat.uniroma2.it}

\address{Universit\'e Grenoble I, Institut Fourier, UMR 5582
CNRS-UJF, BP 74, 38402 Saint Martin d'H\`eres c\'edex, France}
\email{mikhail.zaidenberg@ujf-grenoble.fr}

\thanks{{\bf Acknowledgments:} The first and second authors have been supported by the
Italian MIUR Project protocol 2010S47ARA\_005 and by GNSAGA of
INdAM. The third author was supported by the French-Italian cooperation project GRIFGA and by INdAM. The authors thank all Institutions which helped them in this collaboration, including their own Departments.
}


\begin{document}

\begin{abstract} We investigate the possible values for geometric genera of subvarieties in a smooth projective variety. Values which are not attained are called \emph{gaps}.
For curves on a very general surface in $\PP^3$, the initial gap interval was found by Xu (see \cite{Xu1}), and the next one in our previous paper \cite{CFZ}, where also the finiteness of the set of gaps was established and an asymptotic upper bound of this set was found.  In  the present paper we extend some of these results to smooth projective varieties of arbitrary dimension using a different approach.
\end{abstract}

\maketitle

%\tableofcontents

\vfuzz=2pt
\thanks{}

\noindent

%%%%%%%%%%%%%%%%%%%%%%%%%%%%%%%%%%%%%%%%%%%%%%%%%%
%%%%%%%%%%%%%%%%%%%%%%%%%%%%%%%%%%%%%%%%%%%%%%%%%%
%%%%%%%%%%%%%%%%%%%%%%%%%%%%%%%%%%%%%%%%%%%%%%%%%%

\section*{Introduction} We consider the following problem. Let $X$ be a smooth complex projective variety of dimension $n>1$, with an ample divisor $L$. For each positive integer $s<n$, 
describe the set $\mathcal P_{X,s}$ of geometric genera of irreducible subvarieties $V\subset X$ of dimension $s$, and, in particular, the subset $\mathcal P_{X,L,s}\subseteq \mathcal P_{X,s}$
of geometric genera of irreducible complete intersections of $n-s$ hypersurfaces from $\bigcup_{m\geqslant1} |mL|$. The
complement of each of these sets in $\{0\}\cup\N$ is the corresponding
\emph{set of $s$--gaps}, and its maximal intervals are called \emph{$s$--gap
intervals}. For curves on a very general surface $X$ in $\PP^3$ of degree $d$ (i.e., $n=2$, $s=1$) with a natural polarization $\mathcal O_X(1)$ the two sets $\mathcal P_{X,\mathcal O_X(1),1}$ and $\mathcal P_{X,1}$
coincide; the initial gap interval was found in \cite{Xu1} and the next one in \cite{CFZ}. In this case
there exists a maximum $G_d$  for the set of gaps (\cite[Thm.~2.4]{CFZ} and Remark \ref {rem:compareflam2} below). This means that a very general surface of degree $d$ in $\PP^3$ carries a curve of geometric genus $g$ for any $g> G_d$. In the present note we show that the latter remains true for any smooth projective variety, and in particular, for \emph{any} (not just for a very general) smooth surface  of degree $d$ in  $\PP^3$. One of our main  results is the following:

\begin {thm}\label{thm:main} Let $X$ be an  irreducible, smooth, projective variety of dimension $n>1$, let $L$ be a very ample divisor on $X$ and let $s\in\{1,\ldots,n-1\}$. 
Then there is an  integer $p_{X,L,s}$ (depending on $X$, $L$ and $s$) such that 
for any $p\geqslant p_{X,L,s}$ one can find an irreducible subvariety $Y$ of $X$ of dimension $s$ with at most ordinary points of multiplicity $s+1$ as singularities such that $p_g(Y)=p$. Moreover, one can choose $Y$ to be a complete intersection $Y=D_1\cap\ldots\cap D_{n-s}$, where $D_i\in |L|$ for $i=1,\ldots,n-s-1$ are smooth and transversal and $D_{n-s}\in |mL|$ for some $m\geqslant1$ is such that $Y$ has  ordinary singularities of multiplicity $s+1$. 
\end{thm}

Let $Y$ be an irreducible variety of dimension $s$. A point $y\in Y$ is \emph{ordinary of multiplicity} $m$ ($m>1$), if\\
\begin{inparaenum}
\item [(i)] the Zariski tangent space of $Y$ at $y$ has dimension $s +1$, and\\
\item [(ii)]  the (affine) tangent cone to $Y$ at $y$ is a cone with vertex $y$ over a smooth hypersurface of degree $m$ in $\mathbb P^ {s}$.
\end{inparaenum}

An ordinary point of $Y$ is an isolated hypersurface singularity, hence, it is Gorenstein. 

The proof of Theorem \ref{thm:main} is done in Section \ref{S:nfolds}. In Section \ref{S:surf} we deduce an effective upper bound for gaps in the surface case. In Section \ref{S:surfP3}, we focus on smooth surfaces in $\PP^3$, proving in particular that in this case there is no \emph{absolute gap} for geometric genera of curves. That is, for all $d>0$, all non--negative integers are geometric genera for some curves lying on some smooth surfaces of degree $d$ in $\mathbb P^ 3$.

%%%%%%%%%%%%%%%%%%%%%%%%%%%%%%%%%%%%%%%%%%%%%%%%%%
%%%%%%%%%%%%%%%%%%%%%%%%%%%%%%%%%%%%%%%%%%%%%%%%%%
%%%%%%%%%%%%%%%%%%%%%%%%%%%%%%%%%%%%%%%%%%%%%%%%%%

\subsection* {Notation and conventions} We  work over the field of complex numbers and use standard notation and terminology.
In particular, for $X$ a reduced, irreducible, projective variety,
we denote by $\omega_X$ its dualizing sheaf.  We will sometimes
abuse notation and use the same symbol to denote a divisor $D$ on
$X$ and its class in ${\rm Pic}(X)$. Thus $K_X$ will denote a
canonical divisor or the canonical sheaf $\omega_X$. When $Y \subset X$ is a closed subscheme, 
$\mathcal I_{Y/X}$ will denote its ideal sheaf.

%%%%%%%%%%%%%%%%%%%%%%%%%%%%%%%%%%%%%%%%%%%%%%%%%%
%%%%%%%%%%%%%%%%%%%%%%%%%%%%%%%%%%%%%%%%%%%%%%%%%%
%%%%%%%%%%%%%%%%%%%%%%%%%%%%%%%%%%%%%%%%%%%%%%%%%%

\section{Upper bound for gaps}\label{S:nfolds} 

\subsection{Preliminaries} 
%\begin{sit} 
In the sequel, $X$ is an irreducible,  complex projective variety of dimension $n\geqslant 2$.
%, the dualizing sheaf of $X$ is denoted by $\omega_X$. 
We assume usually that $X$ is Gorenstein, so that $\omega_X$ is a line bundle. This holds, in particular, if $X$ has only ordinary singularities.
We set 
\[
p(X):=h^ 0(X,\omega_{X})\,\,\,\mbox{and}\,\,\, q(X):=h^ 1(X,\omega_{X})\,.
\]
For smooth varieties, both $p(X)$ and $q(X)$ are birational invariants. Note that, if $X$ is a smooth surface, then $q(X)$ is the \emph{irregularity} of $X$.

The  \emph{geometric genus} of $X$ is defined as 
\[ 
p_g(X):= p(X'),
\]
where $X'\to X$ is any desingularization of $X$. 
%\end{sit}

\begin{lem}\label{lem:pg} Let $X$ be an  irreducible, smooth projective variety of dimension $n$, and let $Y$ be an  irreducible, effective divisor on $X$.  Assume that $h^ i(X,\omega_X\otimes \mathcal O_X(Y))=0$ for all  $i\geqslant 1$ \footnote{By the Kawamata--Viehweg vanishing theorem this holds provided
$Y$ is nef and big (in particular, for $Y$ ample).}. Then:\\
\begin{inparaenum}
\item [(i)] one has
\[
p(Y)=h^ 0(X,\omega_X\otimes \mathcal O_X(Y))+q(X)-p_g(X)\,,
\]
which is the geometric genus if $Y$ is smooth;\\
\item [(ii)] suppose that  ${\rm Sing}(Y)=\{x_1,\ldots,x_k\}$, where $x_1,\ldots,x_k$ are  ordinary points of $Y$ of multiplicity $n$. Then
\[
p_g(Y)\geqslant p(Y) -k,
\]
and the equality holds if and only if $x_1,\ldots,x_k$ impose $k$ independent conditions to the linear system $|\omega_X\otimes \mathcal O_X(Y)|$, i.e., if and only if the restriction map
\begin{equation}\label{eq:pp}
H^ 0(X,\omega_X\otimes \mathcal O_X(Y))\longrightarrow \bigoplus_{i=1}^ k \mathcal O_{x_i}
\end{equation}
is surjective. 
\end{inparaenum}
\end{lem}

\begin{proof} Part (i) follows  from the \emph{adjunction sequence}
\[
0\longrightarrow \omega_X \longrightarrow \omega_X\otimes \mathcal O_X(Y)\longrightarrow 
\omega_X\otimes \mathcal O_X(Y)\otimes \mathcal O_Y\cong \omega_Y\longrightarrow 0.
\]

As for part (ii),  let $\pi: X'\to X$ be the blow-up of $X$ at $x_1,\ldots,x_k$ with exceptional divisors $E_1,\ldots, E_k$. Set $E=\sum_{i=1}^ k E_i$. The  union $x$  of $x_1,\ldots,x_k$ is  a 0--dimensional subscheme of $X$. The proper transform $Y'$ of $Y$ in $X'$ is smooth and belongs to the linear system $|\pi^ *(\mathcal O_X(Y))\otimes \mathcal O_{X'}(-nE)|$, whereas $\omega_{X'}=\pi^ *(\omega_X)\otimes \mathcal O_{X'}((n-1)E)$. Hence, by (i), one has 
\[
\begin{split}
p_g(Y)=p_g(Y')&= h^ 0(X',\omega_{X'}\otimes \mathcal O_{X'}(Y'))+q(X')-p_g(X') \\
\, &=h^ 0(X',\pi^ *(\omega_{X}\otimes \mathcal O_{X}(Y))\otimes \mathcal O_{X'}(-E))+q(X)-p_g(X)\\
\, &=h^ 0(X,\omega_{X}\otimes \mathcal O_{X}(Y)\otimes \mathcal I_{x/X})+q(X)-p_g(X)\,.
\end{split}
\] 
Now the assertion follows.\end{proof}

%\subsection{} 

\begin{lem}\label{lem:sec} Let $X\subset \mathbb P^ r$ be a non--degenerate, irreducible  projective variety of dimension $n$. Let $x_1,\ldots, x_k\in X$ be general points. 
%Suppose that 
If 
%$r>n+k-1$. T
$k\leqslant r-n=\codim_{\PP^ r} (X)$,
then the scheme theoretical intersection of the linear space $\langle x_1,\ldots, x_k\rangle$ with $X$ is the reduced $0$--dimensional scheme consisting of $x_1,\ldots, x_k$. 
\end{lem}

\begin{proof} The assertion is trivial for $k=1$, so we assume $k\geqslant 2$. For $n=1$ and $k=2$,  this is the classical \emph{trisecant lemma}, to the effect that a general chord of a non--degenerate curve in $\mathbb P^ r$, where $r\geqslant 3$, is not a trisecant (see, e.g., \cite [Example 1.8]{CC} for a simple proof). If $n=1$ and $k>2$, one proceeds by applying induction on $k$ to the projection of $X$ to $\mathbb P^ {r-1}$ from one of the points $x_1,\ldots, x_k$. 

If $n>1$, one proceeds by applying induction on $n$ to the section of $X$ with a general hyperplane containing $\langle x_1,\ldots, x_k\rangle$. \end{proof} 

\subsection{The theorem}

\begin {thm}\label{thm:main-1} Let $X$ be an  irreducible, smooth, projective variety of dimension $n>1$, and let $L$ be a very ample line bundle on $X$. 
Then there is an integer $p_{X,L}$ (depending on $X$ and $L$) such that for all $p\geqslant p_{X,L}$ one can find  an irreducible hypersurface $Y\in \bigcup_{m \geqslant1}|mL|$ with at most ordinary points of multiplicity $n$ as singularities and with $p_g(Y)=p$. 
\end{thm}

\begin{proof} 
Set $ d:= L^ n$. For a positive integer $m$ we denote by $p_m$ the geometric genus of smooth elements in $|mL|$ (which is of course a non--gap). We  show that for $m$ sufficiently large, any integer $p$ in the interval $[p_{m-1}+1, p_m-1]$ is the geometric genus of a hypersurface in $|mL|$ with $p_m-p$ ordinary points of multiplicity $n$ as singularities, which can be taken generically on $X$.

Since $L$ is very ample, by Lemma \ref {lem:pg}--(i) and by the asymptotic Riemann--Roch Theorem \cite [Vol.\;I, p. 21]{L}, we have 
\begin{eqnarray}\label{eq:pd}
p_m & = & \chi(\omega_{X}\otimes \mathcal O_{X}(mL))+q(X)-p_g(X)\\
     & = & h^0(\omega_{X}\otimes \mathcal O_{X}(mL))+q(X)-p_g(X)\nonumber\\
     & = & \frac {m^ n}{n!} d+ O(m^ {n-1})\,. \nonumber
\end{eqnarray}
Hence
\begin{equation}\label{eq:delta}
\delta_m:=p_m-p_{m-1}-1=\frac  {m^ {n-1}}{(n-1)!} d + O(m^ {n-2})\,.
\end{equation}
Theorem \ref{thm:main-1} follows from the: 

\medskip

\noindent {\bf Claim 1}. \emph{There is an integer $m_{X,L}$ (depending on $X$ and $L$) such that for all $m\geqslant m_{X,L}$, for all positive integers $k\leqslant \delta_m$, and for general points $x_1,\ldots x_k$ 
in $X$, one can find an irreducible element $Y\in |mL|$ with ordinary points of  multiplicity $n$ at 
$x_1,\ldots, x_k$ and no other singularity.} 

\medskip

Indeed, suppose that Claim 1 holds. Then the map \eqref{eq:pp} is surjective by the generality of  $x_1,\ldots, x_k$. Thus Lemma \ref {lem:pg}--(ii) implies Theorem \ref{thm:main-1} with 
\begin{equation*}\label{pX}
p_{X,L}:=p_{m_{X,L}-1}. 
\end{equation*}

In turn, Claim 1 is  a consequence of the following
\medskip

\noindent {\bf Claim 2}. \emph{There is an integer $m_{X,L}\geqslant n$ such that for all $m\geqslant m_{X,L}$, one has}
\begin{equation}\label{eq:r1} 
\delta_m \leqslant\dim(|\nu L|)-n\,,
\quad\mbox{where}\quad m=n\nu + \mu\quad \mbox {with}\quad \mu \in \{0, \ldots, n-1\}\, .
%\nu:=\left\lfloor\frac{m}{n}\right\rfloor\,.
\end{equation}

Indeed, assuming that Claim 2 holds, let $x$ be the reduced $0$--dimensional scheme formed by the points $x_1,\ldots, x_k$, and let $\Lambda: =\nu L\otimes  \mathcal I_{x/X}$. 
By Lemma  \ref {lem:sec}, \eqref{eq:r1} ensures that $x$ is the base locus scheme of the linear system $|\Lambda|$. Therefore, by Bertini's theorem the general  $Y\in |\Lambda^{\otimes n}\otimes \mathcal O_X(\mu)|\subset |mL|$  is irreducible having $x_1,\ldots, x_k$ as ordinary points of multiplicity $n$ and no other singularity.
Thus,  Claim 2 implies Claim 1.
\medskip

Finally, we prove Claim 2.

\begin{proof}[Proof of Claim 2]
%\ref {cl:mm}] 
By the asymptotic Riemann--Roch Theorem (cf.\;\eqref{eq:pd}),  one has 
\[
\dim(|\nu L|)=\frac {\nu^ n}{n!}d + O(\nu^ {n-1})\,.
\] Hence, by \eqref {eq:delta}, Claim 2 holds if, for $m \gg 0$, one has
\begin{equation}\label{eq:nu}
n\,m^ {n-1}<\nu^ n\,.
\end{equation}
Since $\frac{m}{n}<\nu+1$, \eqref {eq:nu} 
is  true for $m\gg 0$. 
\end{proof}

This ends the proof of Theorem \ref{thm:main-1}.  
\end{proof}

\begin{rem}\label{rem:asymreal}  As follows from the proof, the upper bound $p_{X,L}$ depends only on the Hilbert polynomial of $\bigoplus_{m\geqslant1} H^0(X, \omega_X \otimes \cO_X(mL))$ and of $\bigoplus_{m\geqslant1} H^0(X,\cO_X(mL))$. The former coincides with the Hilbert function by Kodaira's Theorem.
Assuming that $h^i(X,\cO_X(mL))=0$ for all positive integers $m$ and $i$, it is possible to replace the asymptotic Riemann--Roch theorem with the true Riemann--Roch, which is then purely numerical. This gives in principle an effective bound on the integers $m_{X,L}$ and $p_{X,L}$ in Theorem \ref {thm:main-1} (cf.\; Section \ref {S:surf} for a particular case).
\end{rem}

\noindent \emph{Proof of Theorem \ref{thm:main}.}
With $X$, $L$, $n$, and $s$ as in Theorem \ref{thm:main}, it
suffices to apply Theorem \ref{thm:main-1} to $X'=D_1\cap\ldots\cap D_{n-s-1}$ instead of $X$ and $L|_{X'}$ instead of $L$, where $D_1,\ldots, D_{n-s-1}\in |L|$ are general. 
\qed

%%%%%%%%%%%%%%%%%%%%%%%%%%%%%%%%%%%%%%%%%%%%%%%%%%
%%%%%%%%%%%%%%%%%%%%%%%%%%%%%%%%%%%%%%%%%%%%%%%%%%
%%%%%%%%%%%%%%%%%%%%%%%%%%%%%%%%%%%%%%%%%%%%%%%%%%

\section{
Genera of curves on smooth surfaces}\label{S:surf} In this section we compute  an effective upper bound for gaps of geometric genera of curves on surfaces. 

\bthm\label{thm:upper-bound} Let $S$ be a  smooth, irreducible, projective surface, and $L$ a very ample line bundle on $S$. Set 
\begin{equation*}\label{eq: notation} p:=p_g(S),\quad q:=q(S),\quad d:=L^2,\quad\mbox{and}\quad e:=K_S\cdot L\,.\end{equation*} For  $\epsilon \in \{0,1\}$, set
\begin{equation}\label{eq:Deltae}
\Delta(\epsilon):= 4 (3 + 2 \epsilon) d^2 + 12 de + e^2 - 8 d (p-q),
\end{equation} 
\be\label{eq:na} n_1 =n_1(\epsilon):=
\begin{cases} \, \, 2 &  \,\,\,\mbox{if}\,\,\,\Delta(\epsilon) < 0,\\ 
\left\lceil 4 + \epsilon + \frac{e}{d} + \sqrt{\frac{\Delta(\epsilon)}{d^2}} \right\rceil&  \,\,\,\mbox{if\,} \,\,\,\Delta(\epsilon) \geqslant 0,\end{cases}
\ee
\be\label{eq:nb} n_2 =n_2(\epsilon):= 
\left\lceil \frac{6(p-q) + d (1+\epsilon) + e (2 \epsilon -1)-12}{e + 2d(1+\epsilon)}  \right\rceil\,,
\ee
\be\label{eq:nbb} n_3 := \min\left\{n\in \mathbb N\,| \left\lfloor \frac n2\right\rfloor^ 2d>nd-\frac {d-e}2-1\right\}\,,
\ee
\be\label{eq:vanishing} 
n_4:=\min\left\{n\in \mathbb N\,\vert\, h^1(S,\cO_S(nL))=h^2(S,\cO_S(nL))=0\right\}\, ,
\ee 
and
 \be\label{eq:n0} n_0 =n_0(\epsilon):= {\rm max} \{n_1(\epsilon), n_2(\epsilon),n_3, n_4\} \,.
\ee
Set finally 
\be\label{eq:phi}
\varphi(d,e,n_0)= \frac{1}{2} \left[(n_0-1)((n_0-1)d + e)\right] + 1\,.\ee 
Then for any $g \geqslant \varphi(d,e, n_0)$ the surface $S$ carries
a reduced, irreducible curve $C$ of geometric genus $g$ with only nodes as singularities. 
\ethm

The proof of Theorem \ref{thm:upper-bound} is basically the same as the one of Theorem \ref {thm:main-1} in the case of surfaces, with a slight improvement, based upon the following:

\bthm\label{thm:Ter-lemma} {\rm (\cite[Thm.\ 1.4]{CC},
\cite[Thm.\ 1.3]{CR})} Let $X \subset \PP^r$ be an irreducible,
projective,  non--degenerate variety of dimension $m$.
Assume $X$ is not $k$--weakly defective for a given $k \geqslant 0$ such that
\be\label{eq:ter-lem} 
r \geqslant(m+1)(k+1)\,.
\ee Then, given general points $p_0,\ldots,p_k$ on $X$, the general hyperplane $H$
containing $T_{X,p_0,\ldots,p_k}$\footnote{$T_{X,p_0,\ldots,p_k}$ stands for the linear span of the union of the embedded  tangent spaces $T_{X,p_i}$, $i=0,\ldots,k$.}  is
tangent to $X$ only at $p_0,\ldots,p_k$. Such a hyperplane $H$
cuts out on $X$ a divisor with ordinary double points at
$p_0,\ldots,p_k$ and no further singularities.
\ethm

Recall (see \cite[p.\,152]{CC1}) that a variety $X$ as in Theorem \ref{thm:Ter-lemma} is
said to be \emph{$k$-weakly defective} if, given $p_0,\ldots,p_k\in X$ general points and 
a general hyperplane $H$ containing $T_{X,p_0,\ldots,p_k}$ (i.e., \emph{tangent} to $X$ at $p_0,\ldots,p_k$), 
then $H$ cuts out on $X$ a divisor $H_X$ such that there is a positive dimensional subvariety 
$\Sigma \subseteq {\rm Sing}(H_X)$ containing  $p_0,\ldots,p_k$
($\Sigma$ is then called the {\em contact variety of} $H$).

%%%%%%%%%%%%%%%%%%%%%%%%%%%%%%%%%%%%%%%%%%%%%%%%%%
%%%%%%%%%%%%%%%%%%%%%%%%%%%%%%%%%%%%%%%%%%%%%%%%%%
%%%%%%%%%%%%%%%%%%%%%%%%%%%%%%%%%%%%%%%%%%%%%%%%%%

\subsection{Proof of Theorem \ref{thm:upper-bound}}\label{ss:nota} 
The arithmetic genus of curves in $|nL|$ is
\begin{equation}\label{eq:pa}
p(d,e,n):=\frac{1}{2}n(nd+e)+1\,.
\end{equation} 
For $n\geqslant n_0$ set
\be\label{eq:l} 
l(d,e,n):=\dim(|nL|) =\frac{1}{2}n(nd-e)+p-q\,,\ee 
where the latter equality follows by the Riemann--Roch Theorem 
and \eqref {eq:vanishing}, since we assume $n\geqslant n_0\geqslant n_4$. Consider the
 embedding $$\phi_{|nL|}\colon S\hookrightarrow \PP^{l(d,e,n)}\,.$$ Since  $\phi_{|nL|}$ is an isomorphism of $S$ to its image $S_n$, we may identify $S$ with 
$S_n$. 

Set
\be\label{eq:delta2}
\delta(d,e,n):=p(d,e,n)-p(d,e,n-1)-1=nd-\frac{1}{2}(d-e)-1\,.
\ee 
As in the proof of Theorem \ref {thm:main-1}, we show that  for any $n \geqslant n_0$ and any positive integer  $k\leqslant\delta(d,e,n)-1$, one can find an irreducible curve $C\in
|nL|$ with exactly $k+1$ nodes at general points of $S$ as its only singularities. Then, for any $n \geqslant n_0$, all the integers in the interval
$J_n=[p(d,e,n-1),\,p(d,e,n)]$ are non-gaps.
Since the intervals $J_n$ and $J_{n+1}$ overlap, this proves
Theorem \ref{thm:upper-bound}, because
\begin{equation}\label{eq:phi2}
\varphi(d,e, n_0):=\min (J_{n_0})=p(d,e,n_0-1)\,
\end{equation} is exactly \eqref{eq:phi}.

The proof follows by Proposition \ref{prop:proof-thm} and Lemma \ref{lem:inequality} below (which are of independent interest).

\bprop\label{prop:proof-thm}  Let $S$ be a  smooth, irreducible, projective surface, and $L$ a very ample line bundle on $S$.  Assume that $n \geqslant \max \{n_3,2\}$  and that (with the above notation) the following inequalities hold
\begin{equation}\label{eq: bound} l(d,e,n)\geqslant3(\delta(d,e,n)-1)
\,\end{equation} and
\begin{equation}\label{eq: bound-1} l(d,e,\lfloor n/2\rfloor)\geqslant\delta(d,e,n)+1\,.\end{equation} 
Then for any
$k\in\{0,\ldots,\delta(d,e,n)-1\}$,\\ \begin{inparaenum}
\item[(a)] the smooth surface $S_n\subset \PP^{l(d,e,n)}$ is not $k$-weakly defective, and\\
\item[(b)]  there exists a reduced, irreducible curve $C\in |nL|$
in $S$ with nodes at $k+1$ general points of $S$ and no other singularity. 
\end{inparaenum}
 \eprop

\bproof Let $x_0, \ldots, x_k$ be general points of $S$. 
Inequality (\ref{eq: bound-1}) guarantees that, for any $k\in\{0,\ldots,\delta(d,e,n)-1\}$, one has
$$\dim\,|\cO_S(\lfloor n/2\rfloor L)\otimes \mathcal I_{\{x_0,\ldots, x_k\}/S}|=l(d,e,\lfloor n/2\rfloor)-k-1\geqslant 
l(d,e,\lfloor n/2\rfloor)-\delta(d,e,n)
\geqslant1\,.$$ 
The general curve in $|\cO_S(\lfloor n/2\rfloor L)\otimes \mathcal I_{\{x_0,\ldots, x_k\}/S}|$ 
is reduced and irreducible. Letting $C_1$ and $C_2$ be two different such general curves, and $C_0$ a general member of $L$, we obtain a divisor 
$$C=\varepsilon C_0+C_1+C_2\in |\cO_S(nL)\otimes \mathcal I_{T_{S,x_0,\ldots, x_k}/\mathbb P^ {l(d,e,n)}}|\,,$$ where $\varepsilon\in\{0,1\}$, $\varepsilon\equiv n\mod 2$. 
Since $C$ is reduced, with nodes at  $x_0,\ldots,x_k$, this shows that (a) holds. 

Now (b) follows. Indeed, since $k+1\leqslant \delta(d,e,n)$,
(\ref{eq: bound}) yields (\ref{eq:ter-lem}) with $m=2$ and $r=l(d,e,n)$. Hence Theorem \ref{thm:Ter-lemma} applies, and so, the general curve in $|\cO_S(nL)\otimes \mathcal I_{T_{S,x_0,\ldots, x_k}/\mathbb P^ {l(d,e,n)}}|$ has nodes at $x_0,\ldots, x_k$ and is elsewhere smooth.  This curve is irreducible by Bertini's theorem. Indeed, if $n$ is odd, then $|\cO_S(nL)\otimes \mathcal I_{T_{S,x_0,\ldots, x_k}/\mathbb P^ {l(d,e,n)}}|$ has no fixed component and is not composed with a pencil. Assume that $n$ is even. By \eqref {eq:nbb}, 
$$C_1\cdot C_2=   \frac {n^ 2}4 d>\delta(d,e,n)\geqslant k+1\,,$$
which motivates \eqref {eq:nbb}.
So, the general curve in $|\cO_S(nL)\otimes \mathcal I_{T_{S,x_0,\ldots, x_k}/\mathbb P^ {l(d,e,n)}}|$, being singular only at $x_0,\ldots, x_k$, cannot be of the form $C_1+C_2$, hence it must be irreducible. \eproof

\blem\label{lem:inequality} Let $\epsilon \in \{0,1\}$ be such that $\epsilon \equiv n \pmod{2}$, and let $\Delta(\epsilon)$ be as in \eqref{eq:Deltae}. Then\\ 
\begin{inparaenum}\item[(a)] {\rm (\ref{eq: bound-1})} holds for any
$n\geqslant n_1$, with $n_1$ as in {\rm \eqref{eq:na}};  \\
\item[(b)] if  {\rm (\ref{eq: bound-1})} holds, then also {\rm (\ref{eq: bound})} holds,  provided
that $n \geqslant n_2$, with $n_2$ as in {\rm \eqref{eq:nb}}. 
\end{inparaenum}\elem

\bproof (a) Write $n = 2 t + \epsilon$, with $t \geqslant 1$ since $n \geqslant n_1\geqslant 2$. From \eqref{eq:l} and \eqref{eq:delta2}, 
\eqref{eq: bound-1} reads 
\begin{equation}\label{eq:aiuto1}
t^2 d - t(4d+e) + 2(p- q) - e + (1-2\epsilon) d\geqslant 0\,.
\end{equation}and the discriminat of the left hand side is $\Delta(\epsilon)$ as in \eqref{eq:Deltae}. 

When $\Delta(\epsilon) \geqslant 0$,  \eqref{eq:aiuto1} holds for $t \geqslant \frac{4d+e + \sqrt{\Delta(\epsilon)}}{2d}$, and so (\ref{eq: bound-1}) holds for 
{\small 
$$n \geqslant 4 + \epsilon + \frac{e}{d} +\sqrt{\frac{\Delta(\epsilon)}{d^2}}\,.$$
}If $\Delta(\epsilon) < 0$, then (\ref{eq: bound-1}) holds for any $n \geqslant 2$. This motivates the definition of $n_1$ in \eqref{eq:na} and proves (a).

\smallskip 

\noindent
(b) As above, (\ref{eq: bound}) reads 
\begin{equation}\label{eq:aiuto2}
n^2 d - n(6d+e) + 2 (p - q) + 3(d - e) + 12 \geqslant 0\,.
\end{equation} Moreover, \eqref{eq:aiuto1} reads  
\begin{equation}\label{eq:aiuto1b}
n^2 d - 8nd - 2ne + 8(p- q) + 4(d -  e) + \epsilon (\epsilon d - 2 nd + 2 e) \geqslant 0\,.
\end{equation}  The difference between the 
left hand side in \eqref{eq:aiuto2} and that of  \eqref{eq:aiuto1b} is 
$$2nd + ne - 6(p - q) - (d - e) - \epsilon (\epsilon d - 2 nd + 2 e)+12\,,$$ which is non--negative as soon as 
$$n \geqslant \frac{6(p-q) + d (1+\epsilon) + e (2 \epsilon -1)-12}{e + 2d(1+\epsilon)}\,.$$Assuming (a), this motivates the definition of 
$n_2$ in \eqref{eq:nb} and   proves (b). 
\eproof

\begin{proof}[Proof of Theorem \ref{thm:upper-bound}] The integer $n_0$ in \eqref{eq:n0} satisfies both (a) and (b) in Lemma \ref{lem:inequality}. Hence \eqref{eq: bound}  and \eqref{eq: bound-1} hold, and we can conclude by Proposition \ref{prop:proof-thm}.
\end{proof}

%%%%%%%%%%%%%%%%%%%%%%%%%%%%%%%%%%%%%%%%%%%%%%%%%%
%%%%%%%%%%%%%%%%%%%%%%%%%%%%%%%%%%%%%%%%%%%%%%%%%%
%%%%%%%%%%%%%%%%%%%%%%%%%%%%%%%%%%%%%%%%%%%%%%%%%%

\section{Genera of curves on smooth surfaces in $\PP^3$}\label{S:surfP3} Here we focus on the case $S$ is a smooth surface of degree 
$d \geqslant 4$ in  $\PP^3$. In \cite{CFZ} we considered the case of a very general $S\in |\mathcal O_{\PP^ 3}(d)|$; here we drop this assumption, and simply assume $S$ smooth and $d\geqslant 4$ (the case $d<4$ being trivial for our considerations, because then $S$ carries curves of any genus). As a direct consequence of Theorem \ref{thm:upper-bound}, we have:

\bcor\label{cor: upper-bound}  For any integer $d \geqslant 4$ there exists an integer $c_d$ such that, for any smooth surface $S$ in  $\PP^3$ of degree $d$
and any integer $g \geqslant c_d$, $S$ carries
a reduced, irreducible nodal curve of geometric genus $g$, whose nodes can be prescribed generically on $S$. 
\ecor

One can give an effective upper bound for $c_d$. We keep here the notation of Section \ref {S:surf}. Letting $L=\mathcal O_S(1)$ 
we obtain $$e = d (d-4),\quad q=q(S)= 0, \;\;\; {\rm and} \;\;\; p =p_g(S)=  \frac{1}{6}(d-1)(d-2)(d-3)\,.$$ By Theorem \ref{thm:upper-bound}
one has
\begin{equation}\label{eq:phi-1}
c_d \leqslant \varphi (d, d(d-4), n_0)\,,
\end{equation}  cf.\ (\ref{eq:phi}).
Thus,  we are left to compute $n_0$ as in \eqref{eq:n0}. Since, by Serre duality, $n_4=d-3$, this amounts to compute $n_1$, $n_2$, and $n_3$ as in \eqref {eq:na}, \eqref {eq:nb}, and \eqref {eq:nbb}.  

From \eqref{eq:Deltae} 
we get $$\Delta({\epsilon}) = d \left( - \frac{1}{3} d^3 + 12 d^2 - \frac{1}{3} (104 - 24 \epsilon) d + 8\right)\,.$$ 
The polynomial $\Delta({\epsilon})/d$ has three positive roots 
$$d_1,d_2,d_3 \sim \begin{cases} 0,25, \;2, 89, \; 32, 86 &\quad\mbox{if}\quad \epsilon=0,\\ 0,36, \; 2, \;\;\;\;\;\; 33,64&\quad\mbox{if}\quad \epsilon=1\,.\end{cases}$$
Thus, 
$\Delta(0) \geqslant 0$ for $4 \leqslant d \leqslant 32$  and $\Delta(0) \leqslant 0$ for $d \geqslant 33$, while $\Delta(1) \geqslant 0$ for $4 \leqslant d \leqslant 33$  and $\Delta(1) \leqslant 0$ for $d \geqslant 34$.

Now
\eqref{eq:na}, \eqref{eq:nb}, and (\ref{eq:nbb}) give,  respectively,  
$$n_1(0)= 
\begin{cases} 2 &  \,\,\,\mbox{if}\,\,\, d \geqslant 33\,,\\ 
\left\lceil d + \sqrt{\frac{\Delta(0)}{d^2}}  \right\rceil  &  \,\,\,\mbox{if}\; 4 \leqslant d \leqslant 32\,,\end{cases}\quad\qquad n_1(1) = 
\begin{cases} 2 &  \,\,\,\mbox{if}\,\,\, d \geqslant 34\,,\\ 
\left\lceil d + 1 + \sqrt{\frac{\Delta(1)}{d^2}}  \right\rceil &  \,\,\,\mbox{if}\; 4 \leqslant d \leqslant 33\,,\end{cases}$$

$$n_2(0)= 
\left\lceil d-5 + \frac{6 (d-3)}{d(d-2)} \right\rceil,\quad\qquad n_2(1)  = 
\left\lceil d-5 + \frac{9(d-2)}{d^2} \right\rceil\,,$$ 

\noindent and

\[
n_3(0)=  3 + \left \lfloor \sqrt {2d-6-(4/d)}  \right \rfloor, \quad\qquad n_3(1)= 4 + \left \lfloor\sqrt {2d-2-(4/d)} \right \rfloor\,.
\]

\noindent In particular, for $d\gg 0$, one has $$n_1=2,\quad n_2\sim d-5, \quad n_3\sim\sqrt{d}, \quad\mbox{hence}\quad n_0=n_4=d-4\,.$$ So, by  (\ref{eq:phi}) and (\ref{eq:phi-1}), 
\[
c_d \leqslant \varphi (d, d(d-4), d-4) = \frac{d(d-5) (2d -9)}{2}\sim d^ 3\,.
\]

\begin{rem}\label{rem:compareflam2}\normalfont{ Let ${\rm Gaps}(d)$ be the set of gaps for geometric genera of irreducible curves on $S\in |\mathcal O_{ \PP^3}(d)|$ very general. By \cite[Theorem 2.4] {CFZ}, one has $${\rm Gaps}(4) = \emptyset, \;\; {\rm Gaps}(5) = \{0,1,2\},\quad\mbox{and}\quad {\rm Gaps}(d) \subset \left[0, \; \frac{d(d-1)(5d-19)}{6} -1\right]\;\; \mbox{for}\;\; d \geqslant 6\,.$$
 This is compatible with the results of the present section. A more refined analysis based on \cite[Remark 2.5] {CFZ}, shows that the maximum $G_d$ of ${\rm Gaps}(d)$ goes 
like $G_d=O(d^ {\frac 83})$. It is an open problem to see if this is sharp.
%%%
\begin{comment}
\noindent
(ii) As in Remark \ref{rem:compareflam1}, one could compare the bounds $c_d +1$ and $p_{X_d, \mathcal O_{X_d}(1)}$ coming from Section \ref{S:nfolds}. This in turn reduces to compare the integer $n_0(d)$ found above and the integers $m_{X_d, \mathcal O_{X_d}(1)}$ expressed in terms of the Hilbert polynomials. Unless $d =5$ (in which case $X_5$ is canonical, and so, $h^2(\mathcal O_{X_d}(1)) = h^0(\mathcal O_{X_d})=1$), one has  
$\nu_{X_d, \mathcal O_{X_d}(1)} = 1$. Therefore, $m_{X_d, \mathcal O_{X_d}(1)} \geqslant 2$ is the smallest integer $m$ satisfying 
$$md + \frac{1}{2} (d^2 - 5d) - 1 \leqslant \frac{(d-1)(d-2)(d-3)}{6} \ \frac{1}{2} \nu^2 d - \nu d (d-4)\,,$$where $\nu = \lfloor \frac{m}{2} \rfloor$. 
%{\bf NOTE: is it worthwile to do all computations?}
\end{comment}
%%%
}
\end{rem}

%%%%%%%%%%%%%%%%%%%%%%%%%%%%%%%%%%%%%%%%%%%%%%%%%%
%%%%%%%%%%%%%%%%%%%%%%%%%%%%%%%%%%%%%%%%%%%%%%%%%%
%%%%%%%%%%%%%%%%%%%%%%%%%%%%%%%%%%%%%%%%%%%%%%%%%%

\subsection{Absence of absolute gaps for curves on smooth surfaces in $\PP^3$} We say that an integer $g$ is a \emph{$d$--absolute gap} if there is no irreducible curve with geometric genus $g$ on any smooth surface of degree $d$. We show here that there is no absolute gap at all. 

\begin{thm}\label{thm:abs-gaps} For any positive integer $d$ and for any non-negative integer $g$, there is a smooth surface $S\subset \PP^ 3$ of degree $d$ and an irreducible, nodal curve $C$ on $S$ with geometric genus $g$.\end{thm}

\begin{proof} We may assume $d \geqslant 5$, otherwise the result is well known (cf. e.g. \cite[Prop.\;1.2 and Cor.\;2.2]{CFZ}). 

We set 
\be\label{4000} \ell_{d,n} := l (d, d(d-4), n) = 
\Bigg \{  \begin{array}{ccc}
& \frac {n (n^ 2+6n+11)}6 \,\,\, &\text {if}\,\,\, n<d \\
& \frac {d\big ( 3n^ 2-3n (d-4)+(d^2-6d+11)\big)}{6} -1
 \,\,\, &\text {if}\,\,\, n\geqslant d\,, \\
\end{array}
\ee and 
\be\label{4000b} p_{d,n}:= p(d,d(d-4),n) =\frac {dn(d+n-4)}2+1\,.\ee 
%As recalled Remark \ref{rem:compareflam2} (i), b
By \cite[Thm.\;2.4 and Rem.\;2.5] {CFZ}, for $S \subset \PP^3$ very general one has
$${\rm Gaps}(d) \subset \left[0,\; p_{d,n-1}-\ell_{d,n-1}-1 \right]= \left[0,\;
\frac{d(d-1)(5d-19)}{6} -1\right]\,\,\, \text{if}\,\,\, d>n\geqslant
\sqrt[3]{12d^ 2}\,.$$Plugging $n=d-1$ in this formula, we obtain the desired result 
%is true 
for all 
\[
g\geqslant p_{d,d-2}-\ell_{d,d-2}\,.
\]Take now $n\leqslant d-2$. By  \cite [Theorem 3.1]{CC0}, for a general surface $\Sigma\subset \PP^ 3$  of degree $n$
with $4\leqslant n\leqslant d-2$ and for any $g\in [p_{n,d}-
\ell_{n,d}, p_{n,d}]$ there is  a reduced, irreducible component
$\mathcal V$ of  the Severi variety of complete intersections of $\Sigma$ with surfaces of degree $d$ having 
$\delta= p_{n,d}-g$ nodes as the only singularities. 

Notice that the union of integers in the non--gap intervals 
$$J_{d-1} (n) = [p_{n,d-1} - \ell_{n,d-1}, p_{n,d-1}] \;\; {\rm and} \;\;  J_{d} (n) = [p_{n,d} - \ell_{n,d}, p_{n,d}]$$ is an integer interval 
for $n \leqslant d-2$. To see this, it suffices to observe that 
$$p_{n,d} - \ell_{n,d} \leqslant  p_{n,d-1} + 1 < p_{n,d}\,.$$
The inequality on the right  is trivial. To show the other inequality
$$\ell_{n,d} \geqslant p_{n,d} - p_{n, d-1} -1\,,$$ using \eqref{4000}, \eqref{4000b}, and the fact that $ d \geqslant n+2 > n$, we can rewrite it as 
$$3 d (d - n +2) + (n^2 - 9 n + 26) \geqslant 0\,,$$which holds if $d \geqslant n+2$.

Any  curve $C\in \mathcal V$ is cut out on $\Sigma$ by a surface $S$ of degree $d$. We claim that $S$ can be taken to be smooth. Since the linear system $|\mathcal I_{C/\PP^3}(d)|$ is base point free outside $C$, by Bertini's theorem $S$ can be choosen to be smooth off $C$. Suppose $S$ is singular at a point $p\in C$. Since $C$ has $\delta$ nodes and no other singularity, and it is the complete intersection of $S$ and $\Sigma$, then $p$ is a node of $C$. But the general surface in 
$|\mathcal I_{C/\PP^3}(d)|$ is non--singular at the nodes of $C$, because $|\mathcal I_{C/\PP^3}(d)|$ contains all surfaces of the form $\Sigma+\Phi$, where $\Phi$ is a general surface of degree $d-n$ (so it does not contain the nodes of $C$), and $\Sigma$ is smooth (thus  $\Sigma+\Phi$ is smooth at the nodes of $C$). 

In this way we find nodal curves of any geometric genus 
$g\geqslant p_{4,d}-\ell_{4,d}=0$ on smooth surfaces of degree $d$, proving the assertion. 
\end{proof}

Let us conclude by the following conjecture. 

\medskip

\noindent {\bf Conjecture.} {\em For any smooth, rational variety $X$
of dimension $n+1$, any very ample line bundle $L$ on $X$, any $s\in\{1,\ldots, n-1\}$, and any integer $g\geqslant 0$  there is a smooth hypersurface $D\in|\cO_X(L)|$ carrying an $s$-dimensional subvariety $S\subset D$ of geometric genus $g$. 

In particular, for any $n \geqslant 3$, $d \geqslant 1$, $s\in\{1,\ldots, n-1\}$, and $g\geqslant 0$ there is a smooth hypersurface $D\in |O_{\PP^{n+1}}(d)|$
and a subvariety $S\subset D$ as before. }

\smallskip

One can ask whether the same holds, more generally,  for any smooth Fano  variety  $X$. 

\providecommand{\bysame}{\leavevmode\hboxto3em{\hrulefill}\thinspace}

\end{document}